\documentclass[11pt]{article}
\usepackage{amsfonts,amssymb,amsmath}
\usepackage[mathcal]{euscript}

\usepackage[ansinew]{inputenc}
\usepackage[english]{babel}
\usepackage{amsfonts,amsmath}
\usepackage{latexsym}
\usepackage{amssymb}
\usepackage{setspace}
\usepackage{hyperref}
\newtheorem{thm}{Theorem}[section]

\newtheorem{lem}{Lemma}[section]
\newtheorem{prop}{Proposition}[section]
\newtheorem{defn}{Definition}[section]
\newtheorem{rem}{Remark}[section]

\newenvironment{proof}[1][Proof]{\textbf{#1.} }{\ \rule{0.5em}{0.5em}}
\newcommand{\E}{\mathbb{E}}
\newcommand{\intot}{\int_0^t}

\newcommand{\sm}{{s-}}
\newcommand{\indiq}{{{\bf 1}}}

\newcommand{\sg}{{\rm sign}}

\hypersetup
{
backref=true, %permet d'ajouter des liens dans...
pagebackref=true,%...les bibliographies
hyperindex=true, %ajoute des liens dans les index.
colorlinks=true, %colorise les liens
breaklinks=true, %permet le retour à la ligne dans les liens trop longs
urlcolor= blue, %couleur des hyperliens
linktoc=section
linkcolor= blue, %couleur des liens internes
bookmarks=true, %créé des signets pour Acrobat
}
\begin{document}
\title{On the time inhomogeneous skew Brownian motion }
\date{March 15, 2012}
\maketitle

\author{\begin{center} {S. Bouhadou, Y. Ouknine}\footnote{Corresponding authors
 Email addresses: {\bf ouknine@ucam.ac.ma} (Y. Ouknine), {\bf sihambouhadou@gmail.com} \\ LIBMA Laboratory, Department of Mathematics, Faculty of Sciences Semlalia, Cadi Ayyad University, P.B.O. 2390 Marrakesh, Morocco.\\}$^{,}$ \footnote{This work is supported by Hassan II Academy of Sciences and Technology.\\
Email address:}
\end{center}}

\bigskip
\begin{spacing}{1.1}
\begin{abstract}
%In this paper, we present an explicit construction of the time inhomogeneous skew Brownian motion; the solution of the following equation:
%$$X_t^\alpha=x+ B_t+\intot  \alpha(s)dL_s^0(X^\alpha), \;\; t \geq 0.$$
%Here $ \alpha: \mathbb{R}^+ \rightarrow[-1,1]$ is a Borel function. $B$ is a standard Brownian motion,$L_t^0(X)$ denotes the local time of $X$ at level zero.
This paper is devoted to the construction of a solution for the "Inhomogenous skew Brownian motion" equation, which first appeared in a seminal paper by Sophie Weinryb, and recently, studied by \'{E}toré and Martinez. Our method is based on the use of the Balayage formula. At the end of this paper we study a limit theorem of solutions.
\end{abstract}

\bigskip

\begin{center}
\end{center}
\newpage
\section{Introduction}
The skew Brownian motion appeared in the  $70$ in the seminal work \cite{ito74a} of Itô and McKean as a natural generalization of the Brownian motion. It is a process that behaves like a Brownian motion except that the sign  of each excursion is chosen using an independent Bernoulli random variable of parameter $\alpha$.
\\
As shown in \cite{H-S}, this process is a strong solution to some stochastic differential equation (SDE) with singular drift
coefficient:
\begin{eqnarray}\label{eqcons}
 X_t&=&x+ B_t +(2\alpha-1)L_t^0(X)
 \end{eqnarray}
 where $\alpha\in (0,1)$ is the skewness parameter , $x \in \mathbb{R}$, and  $L_t^0(X)$ stands for the symmetric local time at $0$.\\
The reader may find many references concerning the homogeneous skew Brownian motion and various extensions in the literature. We cite Walsh \cite{walsh78a}, Harisson and Shepp \cite{H-S}, LeGall \cite{legall} and Ouknine \cite{Ouknine}.
\\
A related stochastic differential equation, introduced by Weinryb \cite{Weinryb} is:
\begin{eqnarray}\label{eqpr}
X_t^\alpha=x+ B_t + \intot  (2\alpha(s)-1) \;dL_s^{0}(X^\alpha), \; \;\; t \geq 0
\end{eqnarray}
where $(B_t)_{t \geq 0}$ a standard Brownian motion on some filtered probability space $(\Omega, \mathcal{F}, \mathbb{P})$, $ \alpha:\mathbb{R}^{+}\rightarrow [0,1]$ is a Borel function and $L^0(X^\alpha)$ stands for the symmetric local time at $0$ of the unknown process $X^\alpha$.\\
The process $X^\alpha$ will be called "time inhomogeneous skew Brownian motion" (ISBM in short). Of course, this equation is an extenstion of the skew Brownian motion. \\
In \cite{Weinryb}, it was shown that there is pathwise uniqueness for equation (\ref{eqpr}) but in \cite{Weinryb} the local time appearing in the equation is standard right sided local time, so that the function $\alpha$ is supposed to take values in $]- \infty,1/2]$ . As is well known, weak existence combined  with pathwise uniqueness, establishes existence and uniqueness of a strong solution to (\ref{eqpr}), via the classical result of Yamada and Watanabe. So, Our purpose in this paper, is to give an explicit construction of the solution of (\ref{eqpr}) by approximating the function $\alpha$ by a sequence of piecewise constant functions $(\alpha_n)$. In order to treat the simple case of a given piecewise constant, we are inspired by a construction given by \'{E}toré and Martinez \cite{etore}, but our proof is totally different. % kind of random flipping for excursions that come from an independent standard reflected  Brownian motion $|B|$:
Instead of trying to show that our construction preserves the Markov property and that the constructed process satisfies (\ref{eqpr}), we use the Balayage formula: the key of "first order calculus".  After its first appearance in Azéma and yor \cite{Yor}, it was later studied extensively in a series of papers as \cite{karoui},\cite{stri} and \cite{sigma}. Note that the point of our departure in this sense, is an interesting observation of Prokaj (see Proposition $3$ \cite{prokaj}), his work is strongly related to the work of Gilat. \cite{gilat}.
\\In \cite{gilat}, the author proved that every nonnegative submartingale is equal in law to the absolute value of a martingale $M$. Barlow in \cite{Barlow-2} gives an explicit construction of the martingale $M$ but for a remarkable class of submartingales. We will show that this result of Barlow is a direct consequence of the Balayage formula.\\
The paper is organized as follows: The second section, we start it with the progressive version of the Balayage formula and we show how to deduce from it a generalization of the observation of Prokaj, in the same section we give a simple proof of the result of Barlow \cite{Barlow-2}. Section $3$ is devoted to the construction of a weak solution of equation (\ref{eqpr}) with a piecewise constant function. Extension of the above result to general case where $\alpha$ is a Borel function is the subject of same section. At end of this work, we study the stability of the solutions of equation (\ref{eqpr}) by using the Skorokhod representation theorem. These result was obtained by \'{E}toré and Martinez \cite{etore} but under some monotonicity assumptions.
\subsection{Preliminaries}
The ISBM has many interesting and sometimes unexpected properties see Etoré and Martinez \cite{etore}. So, The main facts that we use in this paper will be summarized in this section.\\
\textbf{Notation}
\\
For a given semimartingale $X$, we denote by $L_t^0(X)$ its symmetric local time at level $0$. The expectation $\mathbb{E}^x$ refers to the probability measure $\mathbb{P}^x$ under which $X^\alpha_0=x,$ $\mathbb{P}$-a.s.\\
If $k$ is a measurable bounded process, $^p k$ will be denote the predictable projection of $k$. \\
$(\sigma_t)$ is the shift operator acting on time dependent functions as follows:
$ \alpha \circ \sigma_t(s)=\alpha(t+s)$.
 %Let $\alpha: \mathbb{R}^+ \rightarrow [-1,1]$ a borel function. The following fundamental facts are the key of many considerations in this paper.L
\begin{prop}[see \cite{RY}]
Assume $(1)$ has a weak solution $X^{\alpha}$. Then under $\mathbb{P}^0$,
$$(|X_t^{\alpha}|)_{t \geq 0}\stackrel{\mathcal{L}}{\sim}(|B_t|)_{t \geq 0}.$$
\end{prop}
\bigskip
In the introductory article \cite{Weinryb}, it is shown that there is pathwise uniqueness for the weak solutions of  equation (\ref{eqpr}).
\begin{thm}
Pathwise uniqueness holds for the weak solutions of equation  (\ref{eqpr}).
\end{thm}
\begin{defn}
For $t>s$, $x,y\in\mathbb{R}$, we set
\begin{eqnarray*}
p^{\alpha}(s,t;x,y)&= & \int_0^{t-s}\frac{1+\mathrm{sgn}(y)(2 \alpha-1)\circ\sigma_s(u)}{2}\frac{|y|}{\pi}\frac{e^{-\frac{y^2}{2(t-(s+u))}}}{\sqrt{u}(t-s-u)^{3/2}}e^{-x^2/2u}du\\
&+&\frac{1}{\sqrt{2\pi (t-s)}}\exp{-\frac{(y-x)^2}{2(t-s)}} - \exp{-\frac{(y+x)^2}{2(t-s)}} \indiq_{\{xy>0\}}.
\end{eqnarray*}
\end{defn}
\begin{prop}[See \cite{etore}]\label{prop-markov}

Let $X^\alpha$ a strong solution of (\ref{eqpr}) corresponding to the Brownian motion $B$.

i) The process $X^\alpha$ is a Markov process

ii) For all $x,y\in\mathbb{R}$ we have
\begin{equation}
\label{eq-Psx}
\mathbb{P}^{s,x}(X^\alpha_t\in dy)=p^\alpha(s,t,x,y)dy.
\end{equation}
\end{prop}
The next results shows a Kolmogorov's continuity criterion for $X^\alpha$ uniform with respect to the parameter function $\alpha(.)$.
\begin{prop}[see \cite{etore}]\label{kolmo}
There exists a universal constant $C>0$ such that for all $\epsilon \geq 0$ and $t\geq 0$
$$\mathbb{E}^{x}|X_{t+ \epsilon}^{\alpha}-X_{t}^\alpha|^4 \leq C \epsilon^2.$$
\end{prop}
\section {Main results}\label{main result}
\subsection{Some results in the case  where $\alpha$ is constant}
If $M$ is martingale, $k$ is a predictable process such that $\intot k_s^2d\langle M,M\rangle_s< \infty$, we saw how stochastic integration allows to construct a new martingale namely $ \intot k_s dM_s$ with increasing process $\intot k_s^2 d\langle M,M\rangle_s$. The study of the analogous for the local time at $0$ is strongly related to the first order calculus(see e.g.\cite{Yor} ).
\\
\bigskip
At the beginning, let
 $Y=(Y_{t},\;t \geq 0)$ be a continuous $\mathcal{F}_{t}$-semimartingale issued from zero. For every $t>0$ we define
$$\gamma_{t}=sup\{s \leq t\;:\; Y_{s}=0\},$$
with the convention $\sup(\emptyset)=0$, hence in particular $\gamma_{0}=0$. The random variables $\gamma_{t}$ are clearly not stopping times since they depend on the future.
\\
Before stating and proving our main theorem, we shall need a powerful result (see \cite{meyer})
\begin{prop}(Balayage Formula)
\begin{enumerate}
\item[(i)] Let $Y$ be a continuous semimartingale, if $k$ is a bounded progressive process,
then
 $$k_{\gamma_{t}}Y_{t}=k_{0}Y_{0}+\int_{0}^{t}\;^pk_{\gamma_{s}}\;dY_{s}+R_t,$$
 where $R$ is a process of bounded variation, adapted, continuous such that the measure $dR_s$ is carried by the set $\{Y_s=0\}$.
\end{enumerate}
\end{prop}
\begin{rem}
The last zero of continuous processes plays an essential role in the balayage formula: this fact is quite surprising since such random time is not a stopping time and hence falls outside the domain of applications of the classical theorems in stochastic analysis.
\end{rem}

Let $W$ be a Wiener process, it is well known (see. e.g. \cite{RY}) that
\begin{eqnarray}\label{fold}
|W_t|&=&\beta_t+ \sup(- \beta_s)
\end{eqnarray}
where $\beta= \int \sg(W_s) dW_s$ is another Wiener process. In other words, starting with $\beta$, the skorokhod reflection of $\beta$ defined as the right hand side of (\ref{fold}) can be unfolded to a Wiener Process $W$. A similar statement was proved by Prokaj for continuous semimartingales in \cite{prokaj}. Precisely, If $U$ is continuous semimartingale starting from zero and $Y$ denotes the Skorokhod reflection of $U$.  Prokaj has showed that $Y$ can be represented as the reflection $|Y'|$ of an appropriate semimartingale $Y'$ related to $Y$ via the Tanaka equation
$$Y_t=|Y'_t|\;\;\;\; \mbox{and}\;\;\;\; Y'_t= \intot \sg(Y'_s)dU_s.$$
The  proof of this result is based on his key observation corresponding to the case $\alpha=\frac{1}{2}$ in Proposition \ref{prop3} below.
\\
In what follows, we show that this observation is a direct consequence of the balayage formula. We review here this method from Prokaj:
\\

Put $\mathfrak{z}=\{t \geq 0\;:\; Y_t=0\}$, this set cannot be ordered. However, the set $\mathbb{R}^+ \setminus \mathfrak{z}$ can be decomposed as a countable union  $\cup_{ n \in \mathbb{N}} J_n$ of intervals $J_n$. Each interval $J_n$ corresponds to some excursion of $Y$. That is if $J_n=]g_n,d_n[$,\\
$$Y_t \neq 0\;\mbox{for} \; t \in ]g_n,d_n[,\; \mbox{and}\; Y_{g_n}= Y_{d_n}=0. $$
At each $J_n$ we associate a Bernoulli random variable $\xi_n$ which is independent from any other random variables and such that $\mathbb{P}[\xi_n=1]=\alpha$ and\\ $\mathbb{P}[\xi_n=-1]=1-\alpha$. This can be achieved by a suitable enlargement of the  probability field.
\\ Now let $Z$ be the process given by:
  $$Z_t=\sum_{n=0}^{+ \infty} \xi_n \indiq_{]g_n,d_n[}(t).$$

\begin{prop}\label{prop3}
$$Z_tY_t= \intot Z_s dY_s+(2 \alpha-1)L_t^{0}(ZY)$$
\end{prop}
\begin{proof}
To show this proposition, we use the Balayage formula stated in the first part of this paper. First, let us define a process $k$ by :
$$k_t=\sum_{n=0}^{+ \infty} \xi_n \indiq_{[g_n,d_n[}(t).$$
It is obvious  that $ k$ is progressive and bounded. On other hand, we remark that
$$Z_tY_t=k_{\gamma_t}. Y_t,$$ thanks to  the Balayage formula, we have:
\begin{eqnarray}\label{eq3}
Z_tY_t=k_{\gamma_t}. Y_t= \intot \;^p(k_{\gamma_s})\;dY_s+R_t.
\end{eqnarray}
Using the definition of $k$, it is clear that  $^p(k_{\gamma_s})=^p(k_{s})=k_\sm=Z_s$.
Thus, (\ref{eq3}) has the form:
\begin{eqnarray}
Z_tY_t= \intot Z_s\;dY_s+R_t.
\end{eqnarray}
To identify the process $R$, we use a standard approximation of the process
$Z_tY_t- \intot Z_sdY_s.$
For $\epsilon>0$, let us define the following sequence of stopping times:
\begin{eqnarray*}
\tau_0^\epsilon&=&0\\
\tau_{2k+1}^\epsilon&=&\inf\{t> \tau_{2k}^\epsilon\;:\; |Y_t|>\epsilon\}\;\;\;\; k=0,1,2,...\\
\tau_{2k+2}^\epsilon&=&\inf\{t> \tau_{2k+1}^\epsilon\;:\; |Y_t|=0\}\;\;\;\; k=0,1,2,...
\end{eqnarray*}
%For simplicity, we will write only $\tau_n$ instead of $\tau_n^\epsilon$.
Put,
$$Z_t^{\epsilon}=\sum_{k=0}^{+ \infty}Z_t \indiq_{(\tau_{2k+1}^\epsilon,\tau_{2k+2}^\epsilon]}(t).$$
$Z^{\epsilon}$ is constant on every random interval of the form $(\tau_{2k+1}^\epsilon,\tau_{2k+2}^\epsilon]$. The continuity of $Y$ ensures that $Z^{\epsilon}$ is of bounded variation on every compact interval. Hence, $Y$ is Riemann-Stieltjes integrable with respect to $Z^{\epsilon}$ almost surely and
\begin{eqnarray}
Z_t^{\epsilon}Y_t-Z_0^{\epsilon}Y_0- \intot Z_s^{\epsilon}\;dY_s= \intot Y_s\;d Z_s^{\epsilon}.
\end{eqnarray}
As $\epsilon \rightarrow 0$ we have that $Z_t^{\epsilon}\rightarrow Z_t$ for all $t$ almost surely. Since $|Z^{\epsilon}| \leq 1$ the convergence of $Z^{\epsilon}$ implies as well that
$$\intot Z_s^{\epsilon}\;dY_s\rightarrow \intot Z_s\;dY_s$$
in probability for all $t$. The definition of $Z^{\epsilon}$ entails that,
\begin{eqnarray*}
\intot Y_s \;dZ_s^{\epsilon}= \sum_{\stackrel{k}{\tau_{2k+1}<t}}Y_{\tau_{2k+1}^\epsilon} Z_{\tau_{2k+1}^\epsilon}=\epsilon   \sum_{\stackrel{k}{\tau_{2k+1}<t}}Z_{\tau_{2k+1}^\epsilon}\\
\end{eqnarray*}

Let $N(t,\epsilon)$ be the number of upcrossing of the interval $[0,\epsilon]$. So,
$$\intot Y_s \;dZ_s^{\epsilon}=\epsilon \sum_{l=1}^{N(t, \epsilon)}\xi_{k_l}$$
where $\{\xi_{k_l},\;\; l=0,1,...,N(t,\epsilon)\}$ is an enumeration of the $Z_{\tau_{2k+1}}^{\epsilon}$ values.
By a Lévy's résult \cite{levy}, we have $\lim_{\epsilon \rightarrow 0} \epsilon N(t, \epsilon)=\frac{1}{2}L_t^0(|Y|)$. So,
we can apply the Bernoulli law of large numbers with  to get:
$$\epsilon N(t, \epsilon)\frac{ \sum_{l=1}^{N(t, \epsilon)}\xi_{k_l}}{N(t, \epsilon)}\rightarrow \frac{1}{2}L_t^{0}(|Y|).\mathbb{E}[\xi_1]$$

Since $\mathbb{E}[\xi_1]=2 \alpha -1$, we obtain:
$$ \intot Y_s\;d Z_s^{\epsilon} \rightarrow (2 \alpha-1)\frac{1}{2}L_t^{0}(|Y|),$$
since $Z$ has values in the set $\{-1,1\}$, It is clear that
$$\frac{1}{2}L_t^{0}(|Y|)=\frac{1}{2}L_t^{0}(|ZY|)=L_t^{0}(ZY).$$
Hence the result.
\end{proof}
%%%%%%%%%%%%%%%%%%%%%%%%%%%%%%%%%%%%%%%%%%%%%%%%%%%%%%%%%%%%%%%%%%
\bigskip
\subsection{Construction of continuous martingale with given absolute value}
 %We now have all the ingredients to rigourously answer Poblem $1$ raised in the introduction. For this, we will first need to introduce a special class of submartingales which was first introduced by Yor \cite{sigma}.
We devote this subsection to an application of Proposition \ref{prop3}. For this, we will first need to introduce a special class of submartingales whose introduction goes back to Yor \cite{sigma}.
\begin{defn}
 Let $(\Omega, \mathcal{F},(\mathcal{F}_t), \mathbb{P})$ be a filtered probability space. A nonnegative (local) submartingale $(X_t)_{t \geq 0}$ is of class $\Sigma$, if it can be decomposed as \\
 $X_t=N_t+A_t$ where $(N_t)_{t \geq 0}$ and $(A_t)_{t \geq 0}$ are $(\mathcal{F}_t)-$adapted processes satisfying the following assumptions:
 \begin{itemize}
   \item $(N_t)_{t \geq 0}$ is a continuous (local) martingale.
   \item $(A_t)_{t \geq 0}$ is continuous increasing process, with $A_0=0$.
   \item The measure $(dA_t)$ is carried by the set $\{t \geq 0, X_t=0\}$
 \end{itemize}
\end{defn}
The class $(\Sigma)$ contains many well-known examples of stochastic processes such as a nonnegative local martingales, $|M_t|$, $M_t^+$, $M_t^-$ if $M$ is a continuous local martingale. Other remarkable families of examples consist of a large class of recurrent diffusions  on natural scale (such as some powers of Bessel processes of dimension $\delta \in (0,2)$, see \cite{nike}) or of a function of a symmetric Lévy processes; in these cases, $A_t$ is the local time of the diffusion process or of the Lévy process.
\bigskip
\\
In  \cite{gilat}, Gilat proved that every nonnegative submartingale $Y$ is equal in law to the absolute value of a Martingale M. His construction, however, did not shed any light on the nature of this martingale. In the case when $X$ is of class $\Sigma$, one could deduce more interesting result from Proposition \ref{prop3}. The proof is very simple in this case and we give it.
\begin{prop}
If $X \in \Sigma $, then there exists a continuous martingale $M$ such that $X=|M|$
\end{prop}
\begin{proof}
In this proof, we use the same notations as above, let $\alpha= \frac{1}{2}$, by proposition \ref{prop3}, we have:
\begin{eqnarray*}
Z_t X_t = \intot Z_s dX_s= \intot Z_s dN_s+ \intot Z_s dA_s
\end{eqnarray*}
By this argument of the support $supp \;dA_s \subset \{X_s=0\}$, it is clear that
$ \intot Z_s dA_s=0$, putting $M_t=Z_t X_t $, since $Z$ has values in $\{-1,1\}$, we get
$X_t= |M_t|$
\end{proof}

%%%%%%%%%%%%%%%%%%%%%%%%%%%%%%%%%%%%%%%%%%%%%%%%%%%%%%%%%%%%%%%%%%%%%
%CASII
%%%%%%%%%%%%%%%%%%%%%%%%%%%%%%%%%%%%%%%%%%%%%%%%%%%%%%%%%%%%%%%%%%%%%%%

\bigskip

\subsection{The case where $\alpha$ is a piecwise constant function }\label{notation}
Thinking of the case of he proposition \ref{prop3} in which the parameter $\alpha$ is constant, we will content ourselves to find the analogous of this result for a piecewise constant function $\alpha$.
\\

Let $\{\pi: 0=t_0<t_1...<t_i...<t_m=1\}$ be a partition of the interval $[0,1]$.\\
Let $\alpha:[0,1] \rightarrow [-1,1]$  be a r.c.l.l function with constant value in each interval $[t_i,t_{i+1})$. So, $\alpha$ has the form $\alpha(t)= \sum_{i=0}^m \alpha_i\indiq_{\{[t_i,t_{i+1})\}}$  where $\alpha_i \in [0,1]\; $ for all $i=0,...,m$.
\\

Let $ (\xi_n^i)_{n \geq 0}$, $i=0,1,2,...m$ be m independent sequences of independent variables such that,
$\mathbb{P}(\xi_n^i=1)= \alpha(t_i)$ and $\mathbb{P}(\xi_n^i=-1)= 1-\alpha(t_i)$.
\\

Now, put
$$Z_t=\sum_{n=0}^{+ \infty}\sum_{i=0}^{m} \xi_n^i \indiq_{]g_n,d_n[\bigcap [t_i,t_{i+1})}(t).$$
\begin{prop}\label{prop4}
$$Z_tY_t= \intot Z_s dY_s+\intot(2 \alpha(s)-1)dL_s^{0}(ZY)$$
\end{prop}
\begin{proof}
The proof follows the same line as the proof of proposition \ref{prop3}. For a sake of completeness, we give again a proof for the general case.\\
To use the Balayage formula, we must make a suitable choice of $k$. Let $k$ be the process defined by
$$k_t=\sum_{n=0}^{+ \infty}\sum_{i=0}^{m} \xi_n \indiq_{[g_n,d_n[\bigcap [t_i,t_{i+1})}(t).$$
This definition is in fact quite intuitive.
It is obvious  that $ k$ is progressive and bounded. On other hand, we remark that
$$Z_tY_t=k_{\gamma_t}. Y_t,$$ ,as in Proof of the first Proposition, it suffices to apply Balayage formula to get
\begin{eqnarray}
Z_tY_t= \intot Z_s\;dY_s+A_t.
\end{eqnarray}
to identify the process $A$ we use a standard approximation of the process
$Z_tY_t- \intot Z_sdY_s$ by using the same sequence of stopping times as in the proof of proposition \ref{prop3}.
Put,
$$Z_t^{\epsilon}=\sum_{i=0}^{m}\sum_{\stackrel{k=0}{t_i<\tau_{2k+1}<t_{i+1}}}^{+ \infty}\xi_k^{i} \indiq_{(\tau_{2k+1}^\epsilon,\tau_{2k+2}^\epsilon]}(t).$$
$Z^{\epsilon}$ is constant on every random interval of the form $(\tau_{2k+1}^\epsilon,\tau_{2k+2}^\epsilon]$. The continuity of $Y$ ensures that $Z^{\epsilon}$ is of bounded variation on every compact interval. Hence, $Y$ is Riemann-Stieltjes integrable with respect to $Z^{\epsilon}$ almost surely and
\begin{eqnarray}
Z_t^{\epsilon}Y_t-Z_0^{\epsilon}Y_0- \intot Z_s^{\epsilon}\;dY_s= \intot Y_s\;d Z_s^{\epsilon}.
\end{eqnarray}
As $\epsilon \rightarrow 0$ we have that $Z_t^{\epsilon}\rightarrow Z_t$ for all $t$ almost surely. Thus
$$\intot Z_s^\epsilon \;dY_s \rightarrow \intot Z_s \;dY_s$$
By the definition of $Z^{\epsilon}$,
\begin{eqnarray*}
\intot Y_s \;dZ_s^{\epsilon}= \sum_{i=0}^{m}\sum_{\stackrel{k=0}{t_i<\tau_{2k+1}<t_{i+1}}}^{+ \infty}Y_{\tau_{2k+1}^\epsilon} Z_{\tau_{2k+1}^\epsilon}=\epsilon  \sum_{i=0}^{m}\sum_{\stackrel{k=0}{t_i<\tau_{2k+1}<t_{i+1}}}^{+ \infty}Z_{\tau_{2k+1}^\epsilon}\\
\end{eqnarray*}

Let $N(t_i,t_{i+1},\epsilon)$ be the number of upcrossing of the interval $[0,\epsilon]$ between $t_i$ and $t_{i+1}$. So,
$$\intot Y_s \;dZ_s^{\epsilon}=\epsilon \sum_{i=0}^{m}\sum_{l=1}^{N(t_i,t_{i+1},\epsilon)}\xi_{k_l}^{i}$$
where $\{\xi_{k_l}^{i},\;\; l=0,1,...,N(t_i,t_{i+1},\epsilon)\}$ is an enumeration of the $Z_{\tau_{2k+1}}^{\epsilon}$ values between $t_i$ and $t_{i+1}$.
Thus,
\begin{eqnarray*}
\intot Y_s\;dZ_s^{\epsilon}&=& \epsilon\sum_{i=0}^{m}\left[\sum_{l=1}^{N(0,t_{i+1},\epsilon)}\xi_{k_l}^{i}-\sum_{l=1}^{N(0,t_{i+1},\epsilon)}\xi_{k_l}^{i}\right]
\end{eqnarray*}
We can apply the Bernoulli law of large numbers to get:
$$\epsilon N(0,t_{i+1} ,\epsilon)\frac{ \sum_{l=1}^{N(0,t_{i+1},\epsilon)}\xi_{k_l}^i}{N(0,t_{i}, \epsilon)}\rightarrow\frac{1}{2} L_{t_{i+1}}^{0}(|Y|).\mathbb{E}[\xi_1^i]\;\;\;\;\; \forall i=0,1,...,m-1$$
and by the same arguments:
$$\epsilon N(0,t_{i} ,\epsilon)\frac{ \sum_{l=1}^{N(0,t_{i} \epsilon)}\xi_{k_l}^i}{N(0,t_{i}, \epsilon)}\rightarrow \frac{1}{2} L_{t_i}^{0}(|Y|).\mathbb{E}[\xi_1^i]\;\;\;\;\; \forall i=1,2,...,m$$

Since $\mathbb{E}[\xi_1^i]=2 \alpha(t_i) -1$, we get:

$$ \intot Y_s\;d Z_s^{\epsilon}\rightarrow \frac{1}{2}\sum_{i=1}^{m}(2\alpha(t_i)-1)(L_{t_{i+1}}(|Y|)-L_{t_{i}}(|Y|)) =\frac{1}{2}\intot (2 \alpha(s)-1)dL_s^{0}(|Y|),$$
cleary this implies
$$\intot Y_s\;d Z_s^{\epsilon} \rightarrow \intot (2 \alpha(s)-1)dL_s^{0}(ZY).$$
Hence the result.
\end{proof}

%%%%%%%%%%%%%%%%%%%%%%%%%%%%%%%%%%%%%%%%%%%%%%%%%%%%%%%%%%%%%%%%%%%%%%%%%%%%%%%%%%%%%%%%%%%%%%%%%%%%%%%
% Construction of solution
%%%%%%%%%%%%%%%%%%%%%%%%%%%%%%%%%%%%%%%%%%%%%%%%%%%%%%%%%%%%%%%%%%%
\section{Construction of the inhomogenous SBM}\label{Construction}
\subsection{Construction of the inhomogenous SBM with the piecewise constant coefficient $\alpha$}\label{Construction}
In this section, we give a construction of a weak solution of (\ref{eqpr}) obtained by an application of the Proposition \ref{prop4}. We use the same notations and assumptions of the subsection \ref{notation}.\\
The construction is the following: Let $B$ a standard $(\mathcal{F}_t)-$Brownian motion. We define a process $X^\alpha$ by putting
\begin{eqnarray}
\forall t \geq 0\;\;\;\; X_t^{\alpha}(\omega)=Z_t(\omega). |B_t(\omega)|,
\end{eqnarray}
where the process $Z$ is defined by $Z_t=\sum_{n=0}^{+ \infty}\sum_{i=0}^{m} \xi_n^i \indiq_{]g_n,d_n[\bigcap [t_i,t_{i+1})}(t)$, where $\{]g_n,d_n[\}$ is a countable unions of disjoint intervals which covers the set
\\ $\{s \geq 0, \; B_s \neq 0\}$. Consequently, we have the following theorem

\begin{thm}
The process $X^{\alpha}$ is a weak solution of (\ref{eqpr}) with parameter $\alpha$ and starting from $0$.
\end{thm}
\begin{proof}
By Proposition \ref{prop4},
$$X_t^{\alpha}=\intot Z_s\;d|B_s|+ \intot (2 \alpha(s)-1) dL_s^0(X^\alpha).$$
Tanaka's formula implies that:
\begin{eqnarray*}
X_t^{\alpha}&=&\intot Z_s \sg(B_s)dB_s+\intot Z_s\;dL_s^0(B) + \intot (2 \alpha(s)-1) dL_s^0(X^\alpha)\\
&=&\intot Z_s d \beta_s+  \intot (2 \alpha(s)-1) dL_s^0(X^\alpha)
\end{eqnarray*}
where $\beta= \int  sg(B_s)dB_s $ is a Browninan motion. In the last line we have used the fact that the measure $dL_s^0(B)$ is carried by the set $S=\{s\;, B_s=0\}$ and that the process $Z$ is defined  on the complementary of the set $S$,\\ hence $\intot Z_s\;dL_s^0(B)=0$.
Obviously, the process $(\intot Z_s\;d\beta_s)$ is a continuous local martingale. Since $\intot Z_s^2 ds=t$, we deduce that $(\intot Z_s\;dB_s)$ is in fact a Brownian motion, ensuring that $X^\alpha$ satisfies (\ref{eqpr}).
\end{proof}
\subsection{Construction of ISBM with a borel function $\alpha$}
Now, as a natural extension of the construction in the last subsection, we have the following theorem
\begin{thm}\label{thm}
Let $\alpha: \mathbb{R}^{+}\rightarrow [0,1]$ a Borel function and $B$ a standard Brownian motion. For any fixed $x \in \mathbb{R}$, there exists a unique (strong) solution to (\ref{eqpr}). It is a strong Markov process with transition function $p^\alpha(s,t,x,y)$ .
\end{thm}
%%%%%%%%%%%%%%%%%%%%%%%%%%%%%%%%%%%%%%%%%%%%%%%%%%%%%%%%%%%%%%%%%%%%%%%%%%%%%%%%%%
%Skorokhod represent
%%%%%%%%%%%%%%%%%%%%%%%%%%%%%%%%%%%%%%%%%%%%%%%%%%%%%%%%%%%%%%%%%%
The main tool used in the proof is the Skorokhod representation theorem given by the following:
\begin{lem}
Let $(S,\rho)$ be a complete separable metric space, $\{\mathbb{P}_n\;, n \geq 1\}$ and $\mathbb{P}$ be probability measures on $(S, \mathcal{B}(S))$ such that $\mathbb{P}_n \xrightarrow{n \rightarrow+ \infty}\mathbb{P}$. Then on a probability space $(\widehat{\Omega}, \widehat{\mathcal{F}}, \widehat{\mathbb{P}})$, we construct $S-$valued random variables $X_n,$ $\;n=1,2,...$ and $X$, such that:
\begin{itemize}
  \item [(i)] $\mathbb{P}_n=\widehat{\mathbb{P}}_{X_n}\;$ $n=1,2,...$ and $\mathbb{P}=\widehat{\mathbb{P}}_X\;$
  \item [(ii)] $X_n$ converge to $X$, $\widehat{\mathbb{P}}$ almost surely.
\end{itemize}
\end{lem}
We will make use of the following result, which gives a criterion for tightness of sequences of laws associated to continuous processes.
\begin{lem}\label{lema}
Let $X^n(t)$, $n=1,2...,$ be a sequence of d-dimensional continuous processes satisfying two conditions:
\begin{itemize}
  \item [(i)] There exist positive positive constants, $M$ and $l$ such that
  $$\mathbb{E}(|X_0^n|^l) \leq M \;\;\; \mbox{for every} \;\; n=1,2,..$$
  \item [(ii)] There exist positive positive constants, $p, q$, $M_k$, $k=1,2,...$ such that:
  $$\mathbb{E}(|X_t^n-X_s^n|^p) \leq M_k |t-s|^{1+ q}\;\; \mbox{for every} \;\; n \;\;\mbox{and}\;\; t,s \in [0,k].$$
\end{itemize}
Then there exits a subsequence $(n_k)_{k \geq1}$, a probability space $(\widehat{\Omega}, \widehat{\mathcal{F}}, \widehat{\mathbb{P}})$ and d-dimensional continuous processes $\widehat{X}_{n_k},$ $k=1,2,...$ and $\widehat{X}$ defined on it such that
\begin{enumerate}
  \item The laws of $\widehat{X}^{n_k}$ and $X^{n_k}$ coincide.
  \item $\widehat{X}_t^{n_k}$ converges to $\widehat{X}_t$ uniformly on every finite time interval $\widehat{\mathbb{P}}$ almost surely.
\end{enumerate}
\end{lem}
%%%%%%%%%%%%%%%%%%%%%%%%%%%%%%%%%%%%%%%%%%%%%%%%%%%%%%%%%%%%%%%%%%%%%%%%%%%%%%%%%%%%
\begin{proof}[Proof of theorem \ref{thm}]
 $\alpha(.)$ is a borel function, so , there exists a sequence of piecewise constant functions $\alpha_n$ such that $ \lim_{n \rightarrow+ \infty}\alpha_n(t) =\alpha(t),\;$ \\$\forall t \in [0,1]$. Corresponding to such sequences $\alpha_n$, we introduce the corresponding sequences $(X^n)_{n \geq 0}$ which are solutions to equation (\ref{eqpr}) constructed as in the beginning of this section, thus
$$X_t^n=x+B_t+\intot (2\alpha_n(s)-1)dL_s^0(X^n)\;\;\;\;  \forall n \in \mathbb{N}$$

By Lemma \ref{lema} and Proposition \ref{kolmo},
%\begin{eqnarray}
%\E(X_t^n-X_s^n)^4+\E(B_t-B_s)^4 \leq C |t-s|^2.
%\end{eqnarray}
It is clear that the family $(X^n,B)$ is tight. Then there exist a probability space $(\widehat{\Omega},  \widehat{\mathcal{F}},\widehat{\mathbb{P}})$ and a sequence $(\widehat{X}^n,\widehat{B}^n)$ of stochastic processes defined on it such that:
\begin{itemize}
  \item {[$P.1$]}\;The laws of $(X^n,B)$ and $(\widehat{X}^n,\widehat{B}^n)$ coincide for every $n \in \mathbb{N}$.
  \item {[$P.2$]}\;There exists a subsequence $(n_{k})_{k \geq 0}$ such that:
  $(\widehat{X}^{n_k},\widehat{B}^{n_k})$ converge to $(\widehat{X},\widehat{B})$ uniformly on every compact subset of $\mathbb{R}^{+}$ $\widehat{\mathbb{P}}-$a.s.
\end{itemize}
If we denote $ \widehat{\mathcal{F}}_t^n=\sigma\{\widehat{X}_s^{n},\widehat{B}_s^{n}\;; \; s \leq t\}$ and $\widehat{\mathcal{F}}_t=\sigma\{\widehat{X}_s,\widehat{B}_s\;; \; s \leq t\}$, then
$(\widehat{B}^n, \widehat{\mathcal{F}}^n)$ and $(\widehat{B}, \widehat{\mathcal{F}})$ are Brownian motions.  According to property $[P.1]$ and the fact that $X_t^{n}$ satisfies equation (\ref{eqpr}), it's can be proved that $$\mathbb{E }\left|\widehat{X}_t^n-x-\widehat{B}_t^n-  \intot \alpha_n(s)dL_s^{0}(\widehat{X}^{n})\right|^{2}=0.$$
In other words, $\widehat{X}_t^n$ satisfies the SDE:
$$\widehat{X}_t^n=x+\widehat{B}_t^{n}+ \intot (2\alpha_{n}(s)-1)dL_{s}^{0}(\widehat{X}^n).$$
%So, there exist a subsequence $n_k$ such that:
%$$(X^{' n_k},\overline{X}^{n_k},\overline{B}^{n_k}))\xrightarrow[ucp]{L^2}(X',\overline{X},\overline{B})$$
On one hand, from the fact that $(|\widehat{X}_t^{n_k}|)_{t \geq 0}\stackrel{law}{=}(|\widehat{B}_t^{n_k}|)_{t \geq 0}$ and condition $[P.2]$, we deduce that $\;(|\widehat{X}_t|)_{t \in[0,1]}\stackrel{law}{=}(|\widehat{B}_t|)_{t \in [0,1]}$. Thus, $(|\widehat{X}_t|)_{t \in[0,1]}$ is a semimartingale and admits a symmetric local time process $L_{.}(|\widehat{X}|)$.
\\
A consequence of the   Tanaka formula is that:
$$|\widehat{X}_t^{n_{k}}|=|x|+\intot \sg(\widehat{X}_s^{n_{k}})d\widehat{B}_s^{n_k}+L_t^{0}(|\widehat{X}^{n_{k}}|)\;\;\;\;\;(\mbox{with}\;\; \sg(0)=0).$$
As $|\widehat{X}|$ is a reflected Brownian motion we have
\begin{eqnarray}\label{eqcon1}
|\widehat{X}_t|=|x|+\widetilde{B}_t+L_t^0(|\widehat{X}|)
\end{eqnarray}
for some Brownian motion $\widetilde{B}$. By using property $[P.2]$ it holds that
$$\int_0^{.} \sg(\widehat{X}^{n_{k}})d\widehat{B}^{n_k}\xrightarrow[ucp]{L^2}\int_0^{.} \sg(\widehat{X})d\widehat{B} .$$
 Thus,  from the a.s uniform convergence of $(\widehat{X}_t^{n_{k}})_{t \in [0,1]}$ towards $(\widehat{X}_t)_{t \in [0,1]}$ and the dominated convergence for stochastic integrals (see e.g. \cite{RY}) we can see that there is a finite variation process $A$ such that:

$$\sup_{s\in [0,1]}|L_s^{0}(\widehat{X}^{n_{k}})-A_s|\xrightarrow[n \rightarrow \infty]{\mathbb{P}}\;0,$$
and that
$$\sup_{s\in [0,1]}||\widehat{X}_s^{n_k}|-(|x|+ \int_{0}^{s} \sg(\widehat{X}_u)d\widehat{B}_u+A_s)|\xrightarrow[n \rightarrow \infty]{\mathbb{P}}\;0.$$
Consequently,
\begin{eqnarray} \label{eqcon2}
|\widehat{X}_t|=|x|+\intot \sg(\widehat{X}_s)d\widehat{B}_s+A_t
\end{eqnarray}
Using (\ref{eqcon1}) and (\ref{eqcon2}) and uniqueness of the Doob decomposition of a semimartingale :
$$|\widehat{X}_t|=|x|+\intot \sg(\widehat{X}_s)d\widehat{B}_s+L_t^{0}(|\widehat{X}|).$$
Note that we have proven that
$$\left[\sup_{s \in[0,1]}|L_t^{0}(\widehat{X}^{n_k})-L_t^{0}(|\widehat{X}|)|> \epsilon\right]\xrightarrow[n \rightarrow \infty]{P}\;0$$
Using the fact that $\widehat{X}^{n_k}$  is a Markov process with transition family (t.f.) $p^\alpha(s,t,x,y)dy$ (see  \ref{prop-markov}) combined with property $[P.2]$ yields $\widehat{X}^{n_k}\xrightarrow{ucp}\; \widehat{X}$, it is obvious that $\widehat{X}$ is also a Markov process with the same t.f. Now we proceed to the proof of that $\widehat{X}$ is a solution to (\ref{eqpr}). We follow \'{E}toré and Martinez \cite{etore} rather closely.
Hence we may proceed just as in the proof of [Theorem $5.6$ \cite{etore} ]. In fact, by some calculus, we get
$$\E(\widehat{X}_t| \mathcal{F}_s)=\widehat{X}_s+ \int_{0}^{t-s}(2\alpha-1) \circ\sigma_s(u) \frac{e^{-\frac{|\widehat{X}_s|^2}{2u}}}{\sqrt{2 \pi u}}du.$$
Note that $\widehat{X}$ is a Markov process and $|\widehat{X}|$ is a reflected Brownian motion. So that for $s< t$ :
$$\mathbb{E}^0\left(\intot (2\alpha(u)-1)  dL_u^{0}(\widehat{X})|\mathcal{F}_s\right)=\int_{0}^{s}(2\alpha(u)-1) dL_u^{0}(\widehat{X})+ \mathbb{E}^0\left(\int_{s}^{t} (2\alpha-1)(u) dL_u^{0}(\widehat{X})|\mathcal{F}_s\right).$$
but,
\begin{eqnarray*}
\mathbb{E}^0\left(\int_{s}^{t} (2\alpha-1)(u) dL_u^{0}(\widehat{X})|\mathcal{F}_s\right)&=&\mathbb{E}^{\widehat{X}_s}\left(\int_{s}^{t} (2\alpha-1)\circ\sigma_s(u) dL_u^{0}(\widehat{X}^{(2\alpha-1) \circ\sigma_s})\right)\\
&=&\mathbb{E}^{\widehat{X}_s}\left(\int_{s}^{t} (2\alpha-1)\circ\sigma_s(u)  dL_u^{0}(|\widehat{B}|)\right)\\
&=& \int_{0}^{t-s}(2\alpha-1) \circ\sigma_s(u) \frac{e^{-\frac{|\widehat{X}_s|^2}{2u}}}{\sqrt{2 \pi u}}du.
\end{eqnarray*}
Combining these facts ensures that $\{\widehat{X}_t- \intot (2\alpha-1)(u)dL_u^{0}(\widehat{X})\;:\; t \geq 0\}$ is a $(\mathcal{F}_t)$ local martingale. Since $\langle \widehat{X}\rangle_t= \langle|\widehat{X}|\rangle_t=t$, we deduce that, \\
$\{\widehat{X}_t- \intot (2\alpha-1)_u dL_u^{0}(\widehat{X})\;:\; t \geq 0\}$ is a $(\mathcal{F}_t)$-martingale Brownian motion, ensuring that $\widehat{X}$ satisfies (\ref{eqpr}).
\end{proof}
%%%%%%%%%%%%%%%%%%%%%%%%%%%%%%%%%%%%%%%%%%%%%%%%%%%%%%%%%%%%%%%%%%%%%%%%%%%%%%%%%%%%%%%%
\section{Stability of the solution}\label{stability}
%in this section, we state the assumptions. Let $\{\pi_n:0=t_0^n<t_1^n...<t_i^n<...t_n^n=1, \; n\geq 0\}$ a sequence of partitions over $[0,1]$. Assume that
%$$\sup_{0 \leq i\leq n-1}|t_{i+1}^{n}-t_{i}^{n}|\xrightarrow{n \rightarrow +\infty}0.$$

Another key property of the solutions of (\ref{eqpr}) is the following stability result which follows from an application of  Skorokhod Representation theorem.

\begin{thm}
Let $\{\alpha_n(t),\alpha(t): [0,1]\rightarrow [0,1]\}$ be a family of borel functions. Assume that $\lim_{n\rightarrow + \infty}\alpha_n(t)=\alpha(t), \; \; \forall t \in [0,1]$. If we denote $X^{n}$ the solution of (\ref{eqpr}) corresponding to $\alpha_n$, then the following result holds:
$X^{n} \xrightarrow[ucp]{L^2}X$, where $X$ is the unique solution of (\ref{eqpr}) corresponding to $\alpha$.
\end{thm}
\begin{proof}
Suppose that the conclusion of our theorem is false, then there exist a positive number $\delta$  and a subsequence $n_k$ such that
$$\inf_{n_k}\mathbb{E}[\sup_{0\leq s \leq 1} |X_s^{n_k}-X_s|^2] \geq  \delta.$$
According to lemma \ref{lema}, the family $Z^n=(X^n,X,B)$ satisfies conditions (i) and (ii) with $p=4$ and $q=1$.
%\begin{eqnarray}
%\E(X_t^n-X_s^n)^4+\E(X_t^n-X_s^n)^4 \leq C |t-s|^2.
%\end{eqnarray}
By Skorokhod selection theorem, there exists a filtered probability space $(\Omega,  \mathcal{F}^{'},\mathbb{P}^{'} , (\mathcal{F}^{'}_t))$ carrying a sequence of stochastic processes
$\overline{Z}^{n_k}$  denoted by
$\overline{Z}^{n}=(X^{' n},\overline{X}^{n},\overline{B}^{n})$
with the following properties:
\begin{itemize}
  \item {[P.1.]}
  % $$(X^{' n_k},\overline{X}^{n_k},\overline{B}^{n_k})\stackrel{law}{=}(X^n,X,B).$$
  $$Z^{n}\stackrel{law}{=}\overline{Z}^{n}$$
  \item {[P.2.]}There exists a subsequence $(\overline{Z}^{n_{k}})_k$  which converges uniformly to $(X',\overline{X},\overline{B})$.
  %$(X^{' n_k},\overline{X}^{n_k},\overline{B}^{n_k})$ converges to $(X',\overline{X},\overline{B})$ uniformly on every compact subset of $\mathbb{R}^{+}$ $\mathbb{P}-$a.s.
\end{itemize}
%So, there exist a subsequence $n_k$ such that:
%$$(X^{' n_k},\overline{X}^{n_k},\overline{B}^{n_k}))\xrightarrow[ucp]{L^2}(X',\overline{X},\overline{B})$$
Proceeding as in the proof of Proposition \ref{prop4}, we can see that the limiting processes satisfy the following equations.
$$X'_t=x+\overline{B}_t +\intot (2\alpha(s)-1)dL_s^0(X');$$
$$\overline{X}_t=x+\overline{B}_t +\intot (2\alpha(s)-1)(s)dL_s^0(\overline{X}).$$
In other words, $X'$ and  $\overline{X}$ solves equation (\ref{eqpr}). Thus by pathwise uniqueness,  $X'$and $\overline{X}$ are indistinguishable.\\
By uniform integrability, it holds
\begin{eqnarray*}
\delta & \leq & \liminf_{k \in \mathbb{N}} \mathbb{E}[\sup_{0 \leq s \leq 1 } |X_t^{n_k}-X_t|^2]= \liminf_{k \in \mathbb{N}} \widehat{\mathbb{E}}[\sup_{0 \leq s \leq 1 }|X_{t}'^{n_k}-\overline{X}_t^{n_k}|^2] \\
&\leq &  \widehat{\mathbb{E}}[\sup_{0 \leq s \leq 1 } |X'_t-\overline{X}_t|^2]
\end{eqnarray*}
which is a contradiction
\end{proof}

\end{spacing}

\begin{thebibliography}{1}
\bibitem{Yor} J. Azema, M. Yor. En guise d'introduction. Temps locaux, Astérisque n° $52-53$, soc. Math. France $1978$.
\bibitem{bah}K. Bahlali, B. Mezerdi and Y. Ouknine, Pathwise uniqueness and approximation of solutions of stochastic differential equations, sem. de Prob. XXXII, lecture Notes in Maths. 1686, p, 166-187, Springer-Verlag (1997).
\bibitem{Barlow-2}.M. T. Barlow, “Construction of a martingale with given absolute value,” The Annals of Probability, vol. 9, no. 2, pp. 314–320, 1981.
\bibitem{Barlow-1}
M.~T. Barlow, Skew Brownian motion and a one-dimensional stochastic
  differential equation, Stochastics 25 (1988), no.~1, 1--2.
\bibitem{karoui}
N. el. karoui. Temps locaux et balayage des semimartingales, sem. Prob. XIII, Lectures notes in Mathematics 721, Springer Verlag, page 443.
\bibitem{gilat}
D. Gilat. Every non-negative submartingale is the absolute value of a martingale. Ann of Proba., 5, p. 475-481, 1977.
\bibitem{ikeda}N. Ikeda, S. Watanbe, Stochastic differential equations and diffusion processes. North-Holland. Amsterdam (Kodansha Ltd, Tokyo)(1981)
\bibitem{meyer}
P. A. Meyer, C. Stricker and M. Yor, sur une formule de la théorie des balayages, sem. de Prob. XIII, Lectures notes in Mathematics 721, Springer verlag, p. 478.
 \bibitem{etore}
 P. \'{E}toré and M. Martinez on the existence of a time inhomogeneous skew Brownian motion and some related laws. Electron. J. Probab. 17, no. 19, 1-27, (2012).
\bibitem{H-S}
J.~M. Harrison and L.~A. Shepp, On skew Brownian motion, Ann. Probab. 9 , no.~2, 309--313 (1981).
\bibitem{ito74a}
K. It{\^o} and H.P. M{c}{K}ean. Diffusion and Their Sample Paths. Springer-Verlag, 2${}^{\text{nd}}$ edition, 1974.
\bibitem{legall}
J. F. LeGall. One-dimensional stochastic differential equations involving local times of unknown process. Stochastic analysis and application (Swansea 1983) 51-82. Lecture notes in mathematics 1095, Springer-verlag, Berlin 1984.
\bibitem{lejay-2006}
Antoine Lejay, On the constructions of the skew Brownian motion, Probab. Surv. 3, 413--466 (2006).
\bibitem{levy} P. Lévy. Sur les montées des semi-martingales (cas continu), Astérisque, $52-53,$ S.M.F $1978$
\bibitem{nike}
J. Najnudel and A. Nikeghbali, on some universal $\sigma$-finite measures and some extensions of Doob's optional stopping theorem (2009).
\bibitem{Ouknine}
Y.~Ouknine,  "Skew-{B}rownian motion" and derived processes, Teory Probab. Appl. 35 163-169 (1990).
\bibitem{prokaj}
V.~Prokaj, Unfolding the skororhod reflection of a semimartingale, Statistics and probability Letters $79$, $4$ $534$ (2009).
\bibitem{Protter}
P.~E. Protter, \emph{Stochastic integration and differential equations}, Stochastic Modelling and Applied Probability, vol.~21, Springer-Verlag, Berlin, Second edition. Version 2.1, Corrected third printing, 2005.
\bibitem{RY}
D.~Revuz and M.~Yor, \emph{Continuous martingales and {B}rownian motion}, 3rd ed, Springer-Verlag, 1999.
\bibitem{stri}
C. Stricker, Semimartingales et valeur absolue, sem de Prob. XIII, Lecture notes in Maths. 721, Springer Verlag, p. 472.
\bibitem{walsh78a}
J.B. Walsh.
 A diffusion with a discontinuous local time. In Temps locaux, Ast{\'e}risques, pp. 37--45. Soci\'et\'e
  Math\'ematique de France, 1978.
\bibitem{yamada2}
Shinzo Watanabe and Toshio Yamada, On the uniqueness of solutions of stochastic differential equations. II
, J. Math. Kyoto Univ. 11, 553--563 (1971).
\bibitem{Weinryb}
Sophie Weinryb, \'{E}tude d'une \'equation diff\'erentielle stochastique avec temps local, C. R. Acad. Sci. Paris S\'er. I Math.296 , no.~6, 319--321, (1983).
\bibitem{sigma}
Marc Yor, Les inégalités de sous-martingales, comme conséquences de la relation de domination, Stochastics 3, no. 1-15 (1979).
\end{thebibliography}
\end{document}